\documentclass [11pt]{article}
\usepackage{amsthm}
\usepackage{amssymb}
\usepackage{epsfig}
\usepackage{amsthm}
\usepackage{srcltx}
\usepackage[dvipsnames,usenames]{color}
\usepackage{paralist}
\usepackage{amsmath}
\usepackage{amsfonts}
\newcounter{contador}
\newcounter{teoA}

\newtheorem{teoa}[teoA]{Theorem}

\newtheorem{propo}[contador]{Proposition}
\newtheorem{teo}[contador]{Theorem}
\newtheorem{lem}[contador]{Lemma}
\newtheorem{defi}[contador]{Definition}

\theoremstyle{remark}
\newtheorem{nota}[contador]{Remark}

\newcounter{ex}

\newenvironment{example}[1][]{\refstepcounter{ex}\par\medskip
   \noindent \textbf{Example~\theex.} \rmfamily}{\medskip}

\newcommand{\R}{{\mathbb R}}
\newcommand{\C}{{\mathbb C}}
\newcommand{\N}{{\mathbb N}}

\newcommand{\U}{{\cal{U}}}

\title{Parrondo's dynamic paradox for the stability\\ of non-hyperbolic fixed
points}

\author{Anna Cima$^{(1)}$, Armengol Gasull$^{(1)}$ and V\'{\i}ctor Ma\~{n}osa $^{(2)}$
  \\*[.1truecm]
{\small \textsl{$^{(1)}$ Departament de Matem\`{a}tiques, Facultat
de Ci\`{e}ncies,}}
\\*[-.25truecm] {\small \textsl{Universitat Aut\`{o}noma de Barcelona,}}
\\*[-.25truecm] {\small \textsl{08193 Bellaterra, Barcelona, Spain}}
\\*[-.25truecm] {\small \textsl{cima@mat.uab.cat,
gasull@mat.uab.cat}}\\
\\*[-.25truecm] {\small \textsl{$^{(2)}$ Departament de Matem\`{a}tiques,}}
\\*[-.25truecm] {\small \textsl{Universitat Polit\`{e}cnica de Catalunya}}
\\*[-.25truecm] {\small \textsl{Colom 1, 08222 Terrassa, Spain}}
\\*[-.25truecm] {\small \textsl{victor.manosa@upc.edu}}}
\date{}
\begin{document}

% ********************** EN CAS D'ARTICLE *********************
\maketitle
\begin{abstract} We show that for  periodic non-autonomous discrete
dynamical systems,  even when  a common  fixed point for each of the
autonomous associated dynamical systems is repeller, this fixed
point can became a local attractor for the whole system, giving rise
to a Parrondo's dynamic type paradox.
\end{abstract}
% *************************************************************



%\tableofcontents

\section{Introduction and main results}%\label{S-intro}

The study of periodic  discrete dynamical systems is a classical
topic that has attracted the researcher's interest in the last
years, among other reasons, because they are good models for
describing the dynamics of biological systems under periodic
fluctuations whether due to external disturbances or effects of
seasonality, see \cite{BHS,ES1,ES2,FS,Sv,SR1,SR2} and the references
therein.

These $k$-periodic systems can be written as
\begin{equation}\label{E-periodic-system}
x_{n+1}=f_{n+1}(x_n),
\end{equation}
with initial condition $x_0,$ and a set of maps $\{f_m\}_{m\in\N}$
such that $f_m=f_\ell$ if $m\equiv \ell\, (\!\!\!\mod k)$. For
short, the set $\{f_1,\ldots,f_k\}$  will be called \emph{periodic
set.} We also will assume that all $f_m:\mathcal{U}\subset
\R^n\rightarrow \mathcal{U}$ being $\mathcal{U}$ an open set of
$\R^n.$

It is well-known that given a periodic discrete dynamical  system
(\ref{E-periodic-system}), it can be studied via the
\emph{composition map} $f_{k,k-1,\ldots,1}=f_{k}\circ f_{k-1}\circ
\cdots\circ f_1$. For instance, if all maps
$f_m\in\{f_1,f_2,\ldots,f_k\}$ share a common fixed point $p$, the
nature of the steady state $x=p$  can be studied through the nature
of the fixed point $p$ of $f_{k,k-1,\ldots,1}$.  In the same way,
the  attractor of a periodic discrete dynamical system
(\ref{E-periodic-system}) is the union of attractors of some
composition maps, see \cite[Thms. 3 and 6]{FS}.

A specially  interesting case occurs when all the maps in the
periodic set have a  fixed point which is a \emph{global
asymptotically stable} (GAS) and the periodic system has a global
asymptotically stable periodic orbit \cite{A,BHS,CL,ES1}. In this
setting, the simplest situation corresponds to the case when all the
autonomous maps share the same fixed point which is GAS for all of
them and it is also a GAS fixed point for composition map, see
\cite{W}. It is known this is not a general phenomenon, see for
instance Examples \ref{e:lin1} or \ref{e:lin2} of next section.

We will focus on studying the stability of fixed points of
$k$-periodic systems which are common fixed points of all the maps
in the periodic set. We restrict ourselves to this setting because
it is the simplest type of ``periodic orbit'' that a periodic
dynamical system can have.

Notice that given two stable $2\times2$ matrices\footnote{All their
eigenvalues have modulus smaller than 1.}, $A_1$ and $A_2$, it holds
that $|\det(A_i)|<1$ and hence $|\det(A_2
A_1)|=|\det(A_2)\det(A_1)|<1$. As a consequence, the fixed point of
any composition map $f_{k,\ldots,1}$ in $\R^n$ (linear or
non-linear) resulting of the composition of $k$ maps $f_{j}$ with a
common  hyperbolic fixed point, which is asymptotically stable for
all them, must be generically either asymptotically stable or a
saddle, but it can never be repeller. A similar result happens with
maps with a common hyperbolic repeller: generically this point is
either repeller or a saddle for the composition map, but never a
\emph{local asymptotically stable} (LAS) fixed point. Hence, in this
paper, to show that this third possibility may happen in both
situations, we will need to deal with non-hyperbolic fixed points.
We recall the definitions of LAS, GAS, repeller and semi-AS fixed
point in Section~\ref{S-example}.

The so called \emph{Parrondo's paradox} is a paradox in game theory,
that essentially says that {\it a combination of losing strategies
becomes a winning strategy,}  see~\cite{HA,Par}. We will prove that
in the non-hyperbolic case \emph{the periodicity can destroy the
repeller character of the common fixed points, giving rise to
attracting points for the complete non-autonomous system}, showing,
in consequence, the existence of a kind of \emph{ Parrondo's dynamic
type paradox} for periodic discrete dynamical systems. The
phenomenon that we will show  is in a simpler setting than the one
presented in \cite{CLP}, because there the authors combine
periodically one-dimensional maps $f_1$ and $f_2$ to  give rise to
chaos or order.

We start studying the one-dimensional case. The tools for
determining the stability of non-hyperbolic fixed points for
one-dimensional analytical maps are well established, see Section
\ref{ss:estabilitat}.  As we will see, one  the key points is the
computation of the so called {\it stability constants} for studying
the stability of non-hyperbolic fixed points of one-dimensional
non-orientable analytic maps.

Using these tools, in the periodic case we prove the following
result which implies that, contrary to what we will show that
happens in even dimensions,  it is \emph{impossible} to find
\emph{two} one-dimensional maps sharing a fixed point which is
repeller, and such that the composition map has a LAS fixed point.
However, it is possible to find a LAS fixed point when three or more
maps sharing a repeller fixed point are composed (that is for
$k$-periodic systems with $k\geq 3$) giving rise to the  Parrondo's
dynamic paradox.

\begin{teoa}\label{T-laslasreprep} The following statements hold:
\begin{enumerate}
  \item[(a)] Consider two  analytic maps
    $f_{i}:\mathcal{U}\subseteq\R\to\U$, $i=1,2$ having a common fixed
    point $p\in\U$ which is  LAS (resp. repeller). Then, the point $p$ is either
    LAS (resp. repeller)
    or semi-AS for the composition map
    $f_{2,1}$ and both possibilities may happen.
  \item[(b)] There are $k\ge3$ polynomial maps
    $f_{i}:\mathcal{U}\subseteq\R\to\U$, $i=1,2,\ldots,k$ sharing a common fixed
    point $p\in\U$ which is  LAS (resp. repeller) for all them and
     such that  $p$ is a repeller (resp. a LAS) fixed point for the composition
map $f_{k,k-1,\ldots,1}$.
    \end{enumerate}
\end{teoa}

The situations stated in item $(b)$  of the above theorem when
$k=3,$ happen for instance  in Examples \ref{E-f1f2f3}
and~\ref{E-F1F2F3} of Section \ref{S-example}.

Next we  consider the same problem for  planar maps. Again we must
pay attention to the stable/repeller character  for non-hyperbolic
fixed points. We restrict ourselves to maps with \emph{elliptic}
fixed points. These points are fixed points for which the
eigenvalues of the associated linear part lie in the unit circle,
but excluding the values $\pm 1$. For most of them it is possible to
get their  Birkhoff normal form, which permits to compute the so
called \emph{Birkhoff constants} and from them the \emph{Birkhoff
stability constants}. Using these last constants it is possible to distinguish
between LAS and repeller fixed points. We recall all these concepts
in Section~\ref{ss:kam}.

The problem of the local stability of \emph{parabolic} fixed points
(eigenvalues $\pm1$) is much more involved, see \cite{BF,McG,simo,
Sl} for instance, but we do not need to  use them to get our
results.

Now we can state our main result for  the planar case, that presents
again a Parrondo's type paradox in the two-periodic
setting. Some simple examples of maps $f_1$ and~$f_2$ illustrating
it are given in Examples \ref{Ex-dim2-1} and \ref{Ex-dim2-2} of the next
section.

\begin{teoa}\label{T-atratr-dim2} There exist polynomial maps $f_1$ and~$f_2$
in $\mathbb{R}^2$ sharing a common fixed point $p$ which is a LAS
(resp. a repeller) fixed point for both of them, and such that $p$
is repeller (resp. LAS) for the composition map $f_{2,1}$.
\end{teoa}

Combining the  maps that allow to prove  item $(b)$ of Theorem
\ref{T-laslasreprep} and Theorem \ref{T-atratr-dim2} we can prove
the following result that extends these theorems to arbitrary
dimensions.

\begin{teoa}\label{T-Main-A} The following statements hold.

\begin{enumerate}
  \item[(a)] For all $n\geq 1$ there exist  $k\geq 3$ polynomial maps
  $f_i:\mathcal{U}\subseteq\R^{n}\rightarrow\R^{n},$ for
$i\in\{1,\ldots,k\}$, sharing a common fixed point $p$ which is LAS
(resp. repeller) for each map, and such that $p$ is  repeller (resp.
LAS) for the composition map $f_{k,k-1,\ldots,1}$. Furthermore, for
one-dimensional maps ($n=1$), this result is optimal on $k$,  that
is, it is not possible to find only two of such maps such that the
corresponding composition map $f_{2,1}$ satisfies the given
properties.

  \item[(b)] For all $n=2m\geq 2$ there exist $2$ polynomial maps
  $f_1,f_2:\mathcal{U}\subseteq\R^{2m}\rightarrow\R^{2m},$
  sharing a common fixed point $p$ which is LAS
(resp. repeller) for both maps, and such that $p$ is  repeller
(resp. LAS) for the composition map $f_{2,1}$.

\end{enumerate}
\end{teoa}

Although the above theorem is stated for polynomial maps, using
similar techniques, it is easy to construct  examples with the same
properties but with less regularity, say of class $\mathcal{C}^m$
for any $m\ge 6$ (resp. $m\ge 4$)  in item $(a)$  (resp. item
$(b)$). As we will see, this restriction comes from the use of
normal forms and Taylor expansions involving terms until order five
(resp. order three)  in the construction of our examples.

{F}rom its statement it is natural to wonder if item $(a)$ of the
theorem could be improved for $n\ge3,$  odd, taking $k\ge2.$  We
continue thinking on this question.

The paper is structured as follows. In Section \ref{S-example} we
collect all the explicit examples  used to prove Theorems
\ref{T-laslasreprep} and \ref{T-atratr-dim2} and other that help to
contextualize the  problem that we consider. Section
\ref{S-one-dimension} is devoted to prove the results in the
one-dimensional case. In particular, in Subsection
\ref{ss:estabilitat} we give the expression of the first stability
constants and we prove Theorem~\ref{T-laslasreprep}.  Section
\ref{S-dim-2}  is devoted to prove Theorems~\ref{T-atratr-dim2}
and  \ref{T-Main-A}.

\section{Examples and some definitions}\label{S-example}

We start recalling some definitions, see \cite{DEP,E,L}.

\begin{defi}%\label{D-stab}
A fixed point $p$ of a map $f:\U\subset\R^n\to \U$
is said to be:
\begin{enumerate}
  \item[(i)] \emph{Locally asymptotically stable (LAS)} if it is stable
  and it is locally
  attractive. That is, if given  $\varepsilon>0$, there exists $\delta>0$
  such that if  $||x_0-p|| < \delta$ then $||f^n(x_0)-p|| < \varepsilon$
  for all $n\geq 1$ (estable),
  and there exists $\eta>0$ such that if  $||x_0-p|| < \eta$
  then $\lim\limits_{n\to\infty} f^n(x_0)=p$ (attractive). The point
  is \emph{globally asymptotically stable in $\U$ (GAS)}, if it is LAS and
   $\lim\limits_{n\to\infty} f^n(x_0)=p$ for all $x_0\in \U$.
  \item[(ii)] \emph{Repeller} if there exists $\varepsilon_0>0$ such that for any
  $0<\varepsilon<\varepsilon_0$ and for all  $x_0\ne p$
  such that  $||x_0-p|| < \varepsilon$, there exists $n=n(x_0)\in\N$ such
  that $||f^n(x_0)-p||>\varepsilon$.
\end{enumerate}
If $f$ is a one-dimensional map, the fixed point is called:
\begin{enumerate}
  \item[(iii)] \emph{Semi--asymptotically stable}\footnote{
  Also named \emph{saddle-node} or \emph{shunt}.} (semi-AS) from the left (resp. right)
  if given  $\varepsilon>0$, there exists $\delta$
  such that if  $x_0\in(p-\delta,p)$ (resp. $x_0\in(p,p+\delta)$)  then   $|f^n(x_0)-p| < \varepsilon$
  for all $n\geq 1$, and there exists $\eta>0$ such that if
  $x_0\in(p-\eta,p)$ (resp. $x_0\in(p,p+\eta)$) then
  $\lim\limits_{n\to\infty} f^n(x_0)=p$, and there exists $\eta>0$ such that
  if $x_0\in(p,p+\eta)$ (resp. $x_0\in(p-\eta,p)$) then there exists $n\in\N$ such
  that $|f^n(x_0)-p|>\eta$.
\end{enumerate}
\end{defi}

We remark that for invertible maps, instead of definition $(ii)$ it
is simpler to say that a fixed point  $p$ is a repeller for $f$ if
$p$ is an attractor for $f^{-1}$.

Next we collect several examples that illustrate the main results of
this paper. We start with a one-dimensional example that gives the
clue for proving item $(b)$ of Theorem~\ref{T-laslasreprep}.

\begin{example}\label{E-f1f2f3}
Consider the maps:
\begin{align*}
  f_1(x)&=-x+3x^2-9x^3+164x^5, \\
  f_2(x)&=-x+5x^2-25x^3+1259x^5,\\
  f_3(x)&=-x+2x^2-4x^3+33x^5.
\end{align*}
These maps have been chosen using the expressions of the stability
constants given in Proposition \ref{P-stabil-const} of next section,
in such a way that they satisfy $V_3(f_i)=0$ and $V_5(f_i)<0$ for
$i\in\{1,2,3\}.$ Then, they have a LAS fixed point at the origin
(see Theorem \ref{T-orient-rev}). Moreover,  $
f_{3,2,1}(x)=-x+90x^4-48x^5+O(6). $ Computing  the  stability
constants for this map we obtain that $V_3(f_{3,2,1})=0$ and
$V_5(f_{3,2,1})=96>0.$ Hence, using again
Theorem~\ref{T-orient-rev}, we get that the origin is a repeller
fixed point of $f_{3,2,1}$. In the proof of item $(b)$ of Theorem
\ref{T-laslasreprep} we will explain their whole process of
construction.
\end{example}

Clearly, taking  the local inverses of these maps at the origin
until order five we will have an example of  the other situation
stated in item $(b)$ of Theorem~\ref{T-laslasreprep}, which
precisely is  the one that gives rise to the Parrondo's dynamic
paradox. We present it in  the next example.

\begin{example}\label{E-F1F2F3}
Consider the maps
\begin{align*}
  g_1(x)&= T_5(f^{-1}_3(x))=-x+2x^2-4x^3+31x^5,\\
  g_2(x)&=T_5(f^{-1}_2(x))=-x+5x^2-25x^3+1241x^5,\\
  g_3(x)&=T_5(f^{-1}_1(x))=-x+3x^2-9x^3+160x^5,
\end{align*}
where the $f_j$ are the maps given in Example~\ref{E-f1f2f3} and
$T_5$ means the Taylor polynomial of degree 5 at the origin. These
maps have a local repeller at the origin but the composition map
$g_{3,2,1}(x)=-x+90x^4+48x^5+O(6)$ has an attractor at the origin,
because its Taylor polynomial of degree 5 coincide with the one of
the inverse of $f_{3,2,1}.$ In fact the origin is LAS.
\end{example}

\begin{nota} It is interesting to observe that the order in the
periodic set is very important. For instance, in
Example~\ref{E-F1F2F3}, we have seen that the origin of the
composition map $g_{3,2,1}$ is LAS. Nevertheless, by using the
stability constants, it can be seen that the origin of
$g_{1,2,3}(x)=-x+90x^4-72x^5+O(6)$ is repeller.
\end{nota}

Next example shows that even  when two maps have a common GAS fixed
point, the corresponding composition map does not need to have a LAS
fixed point.

\begin{example} Consider the maps
$$
f_1(x)=\left\{
         \begin{array}{ll}
         -x+x^2& x\leq 1,\\
0 & x>1,
         \end{array}
       \right. \,\mbox{  and }\,f_2(x)=\left\{
         \begin{array}{ll}
         -x+2x^2& x\leq 1/2,\\
0 & x>1/2.
         \end{array}
       \right.
$$
It is easy to check that the origin is a GAS fixed point for both of
them. Their corresponding composition map is
$$
f_{2,1}(x)=f_2\circ f_1(x)=\left\{
         \begin{array}{ll}
0& x<\frac{1-\sqrt{3}}{2},\\
x+x^2-4x^3+2x^4& x\in\left[\frac{1-\sqrt{3}}{2},1\right],\\
0 & x>1.
         \end{array}
       \right.
$$
It is not difficult to see that  the origin is a fixed point,
semi-AS from the left for the composition map (see the  Theorem
\ref{T-orient-pres} in next section).  Moreover this map also has
another fixed point at $x=1-\sqrt{2}/2.$ Hence the origin is neither
a global attractor of $f_{2,1}$ nor stable. Gluing  the three pieces
of this example with some suitable bump functions, it is possible to
obtain differentiable or  $\mathcal{C}^\infty$ examples with the
same features.
\end{example}

To end with the one-dimensional examples, and although it is out of
the periodic systems framework, we consider  non-periodic,
non-autonomous system
\begin{equation}\label{E-infinit-GAS}
x_{n+1}=f_{n+1}(x_n),
\end{equation}
with a non-hyperbolic fixed point $p.$ We will show that it is possible
to find maps $f_n$ sharing this common fixed point $p,$ which is a
GAS for each of them, and such that the system has an unbounded
solution.

\begin{example}
We will construct a family of functions $\{f_n\}_{n\geq 0}$ such
that the origin is GAS for each $f_n$ but  the unbounded sequence
$y_n=(-1)^{n}(n+1)$ for $n\geq 0$ is a solution of
\eqref{E-infinit-GAS}.

Consider the following auxiliary map
$g_a(x)=a\left(\exp(-x/a)-1\right)$. Let $a_0$ be the solution of
the equation $g_a(1)=-2$. That is,
$$
a_0=\frac{2}{2W_{-1}\big(-\exp(-1/2)/2\big)+1}\simeq -0.7959,
$$
where $W_{-1}(x)$ is the secondary branch  of the Lambert
$W$-function, \cite{cor, RO}.

Now we take $f_0:=g_{a_0}$. The map $f_0$ has a GAS fixed point at
the origin\footnote{Since $f_0(-\infty,0)\subset(0,+\infty)$ and
$f_0(0,+\infty)\subset(-\infty,0)$ it is enough to check that if
$x>0$, then for $k\geq 1$ we have $f^{2k}_{0}([0,x])\subset
[0,a_k)$, where $\lim_{k\to+\infty} a_k=0^+$; and  if $x<0$, then
for $k\geq 1$ we have $f^{2k-1}_{0}([0,x])\subset [0,b_k)$, where
$\lim_{k\to+\infty} b_k=0^+$.} which is non-hyperbolic, since
$f_0'(0)=-1$, and it satisfies $ f_0(1)=-2$.

In order to construct the maps $f_n$ for $n\geq 1$, we consider the
following linear conjugations
$$
h_n(x):=\frac{(-1)^{n}}{3}\big((2n+3)x+n\big).
$$
These maps are chosen in such a  way that
$h_n(-2)=(-1)^{n+1}(n+2)={y}_{n+1}$ and $h_n^{-1}({y}_n)=1$. Now, we
define
$$
f_n(x):=h_n\circ f_0\circ h_n^{-1}(x), \, n\geq 1.
$$
Obviously each map has a GAS point at the origin, because they are
conjugate to $f_0$. Finally, observe that for all $n\in\N,$
$$f_n({y}_n)=h_n\circ f_0\circ h_n^{-1}({y}_n)=h_n\circ
f_0(1)=h_n(-2)={y}_{n+1}.$$ Hence the unbounded sequence
$\{y_n\}_{n\in\N}$ is a solution of (\ref{E-infinit-GAS}) with
initial condition $x_0=1$.
\end{example}

We continue this section with some simple linear two-dimensional
examples.

\begin{example}\label{e:lin1}
This  first example shows two linear maps such that for each of them
the origin which is GAS (in fact a super-attracting point), but the
origin is a saddle point of the composition map $f_{2,1}.$ Hence the
origin is an unstable steady state of the periodic system, but not repeller.
Set $\mathbf{x}=(x,y)$ and $f_i(\mathbf{x})=A_i \cdot \mathbf{x}^t$,
where
$$A_1=
\left(
    \begin{array}{cc}
       0 &  2\\
      0 & \frac{1}{2}\\
    \end{array}
  \right)
  \,\mbox{ and }\, A_2=
\left(
    \begin{array}{cc}
       \frac{1}{2} &  0\\
      2 &  0\\
    \end{array}
  \right).
$$
Observe that
$\mbox{Spec}\left(A_1\right)=\mbox{Spec}\left(A_2\right)=
\{0,1/2\},$ so the origin is  GAS for the dynamical systems
associated to $f_1$ and $f_2$. The corresponding composition map
associated to the $2$-periodic system is
$f_{2,1}(\mathbf{x})=A_{2,1} \cdot \mathbf{x}^t$ where
$$A_{2,1}:=A_2\cdot A_1=\left(
    \begin{array}{cc}
       0 &  1\\
      0 &  4\\
    \end{array}
  \right).
$$
Since $\mbox{Spec}\left(A_{2,1}\right)=\{0,4\},$  the origin is a
saddle point for $f_{2,1}.$
\end{example}

\begin{example}\label{e:lin2} In \cite{BTT} the authors consider
the maps $f_i(\mathbf{x})=A_i \cdot \mathbf{x}^t$, $i=1,2$ where
$$A_1=\alpha\,
\left(
    \begin{array}{cc}
       1 &  1\\
      0 & 1\\
    \end{array}
  \right)\,\mbox{ and }\,A_2=\alpha\,
\left(
    \begin{array}{cc}
       1 &  0\\
      1 &  1\\
    \end{array}
  \right),
$$
with $|\alpha|<1$. Both maps have the origin as a GAS point because
$\mbox{Spec}\left(A_1\right)=\mbox{Spec}\left(A_2\right)=
\{\alpha\}$.  Then the composition map is
$f_{2,1}(\mathbf{x})=A_{2,1} \cdot \mathbf{x}^t$ with
$$A_{2,1}=\alpha^2\, \left(
    \begin{array}{cc}
       1 &  1\\
      1 &  2\\
    \end{array}
  \right),$$
and it is such that $\mbox{Spec}\left(A_{2,1}\right)=\{\left( 3\pm
\sqrt {5} \right) {\alpha}^{2}/2\}.$ Hence the origin is either GAS
if $|\alpha|<(\sqrt{5}-1)/2\simeq 0.618$, or a saddle point if
$(\sqrt{5}-1)/2<|\alpha|<1$. Again the stability can be lost but
again  no repeller fixed points appear.
\end{example}

Another nice linear example is given
 in \cite[p. 8]{J}. There the author  shows two linear maps, both  with a stable
 focus at the origin, and such
 that the corresponding composition map has again a saddle at the origin.

We end this section  with two planar examples that show that taking
a periodic set with only two elements with non-hyperbolic fixed
points, the stabilities of the common fixed point can be reversed
for the corresponding composition maps. In fact they allow
 to prove Theorem \ref{T-atratr-dim2}. Their constructions are detailed in Section
 \ref{proofTB}.

\begin{example}\label{Ex-dim2-1} Consider the maps
\begin{align*}
f_1(x,y)&=\left(-y+2{x}^{2}+6xy,x-3{x}^{2}+2xy+3{y}^{2}\right),\\
f_2(x,y)&=\left(\frac{x}{2}-\frac{\sqrt{3}}{2}y-x\,(x^2+y^2),
\frac{\sqrt{3}}{2}x+\frac{1}{2}y-y\,(x^2+y^2)\right).
\end{align*}
As we will see in the proof of Theorem \ref{T-atratr-dim2} in
Section \ref{S-dim-2}, the origin is a LAS fixed point for both maps
$f_1$ and $f_2$, because their Birkhoff stability constants are
$V_1(f_1)=V_1(f_2)=-1/2<0.$ However the origin is a repeller fixed
point for the composition map $f_{2,1}$ because
$V_1(f_{2,1})=(3\sqrt{3}-5)/2\simeq 0.098>0$.
\end{example}

Each map $f_i$ is locally invertible in a neighborhood of the
origin. Hence, as we did for passing from Example~\ref{E-f1f2f3} to
Example~\ref{E-F1F2F3}, by taking their local inverses $f^{-1}_i,$
we have maps with a repeller fixed point at the origin such that
their composition map has a LAS fixed point at the origin, giving an
example of the remaining case considered in
Theorem~\ref{T-atratr-dim2}. Anyway, we also include an explicit
independent example.

\begin{example}\label{Ex-dim2-2}
The origin is a repeller fixed point for both maps
\begin{align*}
f_1(x,y)&=\left(-y+\frac{1}{3}x^2-8xy+\frac{5}{3}y^2,x+4x^2-\frac{4}{3}xy-
4y^2\right),\\
f_2(x,y)&=\left(\frac{x}{2}-\frac{\sqrt{3}}{2}y+x(x^2+y^2),
\frac{\sqrt{3}}{2}x+\frac{1}{2}y+y(x^2+y^2)\right),
\end{align*}
because their corresponding Birkhoff stability constants are
$V_1(f_1)=V_1(f_2)=1/2>0$. Now the origin is a LAS fixed point for
$f_{2,1}$, because $V_1(f_{2,1})=3-2\sqrt{3}\simeq -0.464<0$.
\end{example}

\section{One-dimensional maps}\label{S-one-dimension}

\subsection{Stability of fixed points}\label{ss:estabilitat}

In this section we consider one-dimensional analytic maps with a
fixed point that, without loss of generality, we take as the origin
and we denote by $\U$ a neighborhood of this point. As we have
already mentioned, it is clear that, from the view point of the
stability problem for composition, the more interesting maps are the
ones having non-hyperbolic fixed points. A summary of several
results concerning this situation can also be found in \cite{DEP}.
Next we recall some of them and also develop some new results for
the orientation reversing case.

 This first result is well-known and
characterizes the local dynamics at a non-hyperbolic non-oscillatory
fixed points  one-dimensional maps $f$ (i.e. $f'(0)=1$, that is,
when $f$ is \emph{locally orientation preserving}):

\begin{teo}[\cite{DEP}]\label{T-orient-pres}
Let $f$ be a $\mathcal{C}^{m+1}(\mathcal{U})$ function such that
$f(0)=0$, $f'(0)=1$ and
$$
f(x)=x+a_m x^m+O(m+1),\mbox{ with } a_m\neq 0,\,m\geq 2.
$$
Then:
\begin{itemize}
  \item[(a)] If $m$ is even then the origin is semi-AS from the left if
  $a_m>0$ and from the right if $a_m<0$.
  \item[(b)] If $m$ is odd then the origin is  repeller if
  $a_m>0$ and  LAS if $a_m<0$.
\end{itemize}
\end{teo}

The complete study of the local dynamics at a non-hyperbolic
oscillatory fixed point (i.e. when $f(0)=0$ and $f'(0)=-1$, that is
when $f$ is \emph{locally orientation reversing}), is more involved.
In \cite[Thm 5.1]{DEP} a result is given in terms of the derivatives
of the orientation preserving map $f^2=f\circ f$, by using Theorem
\ref{T-orient-pres}. However in \cite[Thm 5.4]{DEP}, to avoid using
these derivatives the authors present a slightly more explicit
expressions obtained using the Fa\`a di Bruno Formula
(\cite{Johnson}). The expressions in \cite[Thm 5.1]{DEP} are closely
related with what we call \emph{stability constants}, that we
introduce below.

Given a  $\mathcal{C}^{\omega}(\mathcal{U})$ function of the form
$$
f(x)=-x+\sum\limits_{j\geq 2} a_j x^j,
$$
one obtains
\begin{equation*}%\label{E-Ws}
f^{2}(x):=f\circ f(x)=x+\sum\limits_{j\geq 3} W_j(a_2,\ldots,a_j)
x^j.
\end{equation*}

If $f$ is not an involution (i.e. $f^{2}\neq\mathrm{Id}$), we define
the \emph{stability constant of order} $\ell\geq3,$ as $V_\ell,$
where
\begin{equation*}%\label{E-stab-cons}
V_3:=W_3(a_2,a_3)\quad\mbox{and}\quad
      V_\ell:=W_\ell(a_2,\ldots,a_\ell) \,\hbox{ if }\, W_j=0,\,j=3,\ldots,\ell-1.
\end{equation*}

Next result shows that the first non-zero stability constant is for
$\ell$ odd and gives the stability of the fixed point.

\begin{teo}\label{T-orient-rev}
Let $f$ be an analytic map in $\mathcal{U}\subseteq\R$ such that
$f(0)=0$, $f'(0)=-1$. If $f$ is not an involution, then there exists
$m\geq 1$ such that $V_3=V_4=V_5=\cdots=V_{2m}=0$ and  $V_{2m+1}\neq
0$. Moreover, if $V_{2m+1}< 0$  (resp. $V_{2m+1}> 0$),  the origin
is LAS (resp. repeller).
\end{teo}

\begin{proof} We start proving that
the first non-zero stability constant has odd order. Suppose, to
arrive to a contradiction, that $f^{2}(x)-x=V_{2m}x^{2m}+O(2m+1)$
with $V_{2m}\ne 0.$ Assume, for instance that $V_{2m}>0.$ Then we
can consider a neighborhood of the origin $\tilde{\U}\subseteq \U$
such that for all $x\in\tilde{\U}\setminus\{0\},$ $f$ is strictly
monotonically decreasing and $f^2(x)-x>0.$

Let $x_0\in\tilde{\U}\setminus\{0\}$ and consider its orbit
$x_n=f^n(x_0).$ We take $x_0$ small enough with
$x_1,x_2,x_3\in\tilde{\U}\setminus\{0\}.$ We have that
$x_2-x_0=f^2(x_0)-x_0>0.$ Since $f$ is decreasing, $f(x_2)<f(x_0),$
that is, $f^2(x_1)<x_1,$ a contradiction.

Now, the theorem is a direct corollary of statement $(b)$ of Theorem
\ref{T-orient-pres}.~\end{proof}

Finally,  we give an expression of some stability constants. It is
clear that the regularity of the function can be weakened in their
computation, because the only needed tools are  the Taylor
expansions at the origin.

\begin{propo}\label{P-stabil-const} Let $f$ be a
 $\mathcal{C}^{12}(\mathcal{U})$ function such that $f(0)=0$,
$f'(0)=-1$. Then the first stability constants are
  \begin{align*}
  V_3=& 2\left(-{a_{{2}}}^{2}-a_{{3}}\right),\\
  V_5=& 2\left(2 {a_{{2}}}^{4}-3 a_{{2}}a_{{4}}-a_{{5}}\right),\\
  V_7=& 2\left(-13 {a_{{2}}}^{6}+18 {a_{{2}}}^{3}a_{{4}}-4 a_{{6}}a_{{2}}-2 {a_{{
4}}}^{2}-a_{{7}}
\right),\\
  V_9=& 2\left(145 {a_{{2}}}^{8}-221 {a_{{2}}}^{5}a_{{4}}+35 {a_{{2}}}^{3}a_{{6}}+
50 {a_{{2}}}^{2}{a_{{4}}}^{2}-5 a_{{2}}a_{{8}}-5 a_{{6}}a_{{4}}-a_{
{9}}
\right),\\
  V_{11}=& 2\left(-2328 {a_{{2}}}^{10}+3879 {a_{{2}}}^{7}a_{{4}}-561 {a_{{2}}}^{5}a_{
{6}}-1263 {a_{{2}}}^{4}{a_{{4}}}^{2}+61 {a_{{2}}}^{3}a_{{8}}+171 {a
_{{2}}}^{2}a_{{4}}a_{{6}}+\right.\\&\left. \quad 55
a_{{2}}{a_{{4}}}^{3} -6 a_{{2}}a_{{10}}- 6 a_{{4}}a_{{8}}-3
{a_{{6}}}^{2}-a_{{11}} \right),
  \end{align*}
\end{propo}

\begin{proof}
The constants have been obtained by computing the first coefficient
of the Taylor expansion of $f^2$. It is easy to check that  any
constant $V_{2m+1}$ contains the monomial $-2 a_{2m+1}$. Hence, once
$V_{2m+1}$ is obtained, the constant $V_{2m+3}$ is computed by
solving $V_{2m+1}=0$ with respect the coefficient $a_{2m+1}$ and
plugging this value in the expression of $W_{2m+3}$.~\end{proof}

We stress that more constants can be easily obtained with very few
computing time.

\subsection{Proof of Theorem \ref{T-laslasreprep}}

Observe that the statement $(a)$ of Theorem \ref{T-laslasreprep} is
a consequence of the following  result.

\begin{teo}\label{T-main} Consider two analytic maps
    $f_{i}:\mathcal{U}\subseteq\R\to\U$, $i=1,2$ having a common non-hyperbolic fixed
    point $p\in\U$ which is LAS (resp. repeller) for both of them. Then  $p$ is either
    LAS (resp. repeller) or semi-AS for the composition map
    $f_{2,1}=f_2\circ f_1$. More precisely,
    \begin{itemize}
        \item[(a)] If one of the maps $f_i$ preserves orientation, then $p$ is LAS (resp. repeller)
         for $f_{2,1}$.
        \item[(b)] If both $f_1$ and $f_2$ reverse orientation, then $p$ can be either a
        LAS (resp. repeller) or a semi-AS fixed point for $f_{2,1}$.
    \end{itemize}
\end{teo}

To prove it, we  introduce  the differentiable normal form of an
analytic map $f$ with a non-hyperbolic fixed point at $0$, which is
given by the next result in \cite{Ch} (see also \cite{AP,ChDD}). A
similar result for $\mathcal{C}^\infty$ maps can be found in
\cite[Thm 2]{T}:

\begin{teo}[K.~Chen, \cite{Ch}]\label{T-Chen}
Let $f$ be an \emph{analytic} diffeomorphism on $\R$. If $f$ is
orientation preserving (resp. reversing), and it is not an
involution, then given any positive integer $\ell$ there exists a
$\mathcal{C}^\ell$ local diffeomorphism $\varphi$ on $\R$ such that
$g=\varphi^{-1}\circ f\circ\varphi$ is in one of the normal forms:
\begin{enumerate}
  \item[(a)] $g(x)=\lambda x$, with $|\lambda|\neq 1$ and $\lambda>0$
  (resp. $\lambda<0$),
  \item[(b)] $g(x)=x+(\pm x)^{m+1}+cx^{2m+1}$ (resp.
  $g(x)=-x\pm x^{m+1}+cx^{2m+1}$),
\end{enumerate}
where $c\in\R$, and $m$ is a positive (resp. positive even) integer.
\end{teo}

Next result classifies the stability of each of these normal forms.
The proof follows from a straightforward application of
Theorems~\ref{T-orient-pres} and \ref{T-orient-rev}.

\begin{lem}\label{L-fn-chen} The following statements hold.
\begin{itemize}
  \item[(a)]  The map $f(x)=x+
x^{m+1}+cx^{2m+1},$ with $0<m\in\N$ and $c\in\R,$ has a semi-AS from
the left fixed point at the origin if $m$ is odd, and a repeller
fixed  point if $m$ is even.
  \item[(b)]The map $f(x)=x-
x^{m+1}+cx^{2m+1},$ with $0<m\in\N$ and $c\in\R,$ has a semi-AS from
the right fixed point at the origin if $m$ is odd, and a LAS fixed
point if $m$ is even.
  \item[(c)]The map $f(x)=-x+
x^{2r+1}+cx^{4r+1},$ with $0<r\in\N$ and $c\in\R,$ has a LAS fixed
point at the origin.
  \item[(d)] The map $f(x)=-x-
x^{2r+1}+cx^{4r+1},$ with $0<r\in\N$ and $c\in\R,$ has a repeller
fixed point at the origin.
\end{itemize}
\end{lem}

To prove Theorem \ref{T-main}, we also need the following result.

\begin{lem}\label{L-lema2}
Let $f$ be an  analytic map in $\mathcal{U}$ which is not an
involution, given by $f(x)=-x+\sum_{j\geq 2} a_j x^j$. Assume that
$a_{2j}=0$ and $V_{2j+1}=0$ for all $j=1,2,\ldots,m.$ Then
$a_{2j+1}=0$ for all $j=1,2,\ldots,m.$
\end{lem}
\begin{proof}
We prove the lemma by induction on $m.$ If $m=1$ then $a_2=V_3=0.$
By Proposition~\ref{P-stabil-const}, since
$V_3=-2(a_2^2+a_3)=-2\,a_3=0$ we get $a_3=0.$

Now assume that the result is true for all $j\le m-1$ and assume
that $a_{2j}=0$ and $V_{2j+1}=0$ for all $j=1,2,\ldots,m.$ In
particular $a_{2j}=0$ and $V_{2j+1}=0$ for all $j=1,2,\ldots,m-1.$
Applying the induction hypothesis we get $a_{2j+1}=0$ for all
$j=1,2,\ldots,m-1.$ Then, $f(x)=-x+a_{2m+1}\,x^{2m+1}+O(2m+2)$ which
implies that $V_{2m+1}=-2\,a_{2m+1},$  because
$f^{2}(x)=x-2\,a_{2m+1}\,x^{2m+1}+O(2m+2).$  Since $V_{2m+1}$ is
zero, also $a_{2m+1}$ must be zero.
\end{proof}

\begin{proof}[Proof of Theorem \ref{T-main}]

We will consider only the situation where $f_1$ and $f_2$ have a LAS
fixed point; the other case follows similarly.

$(a)$ Let $f_1$ and $f_2$ be the two maps, and assume that the
second one preserves orientation. Without loss of generality, we can
take the first one in one of  its normal forms given in
Lemma~\ref{L-fn-chen}, $f_1(x)=\pm x \mp x^{2r+1}+c\,x^{4r+1},$ and
the second one, by Theorem~\ref{T-orient-pres}, as
$f_2(x)=x+a\,x^{2n+1}+O(2n+2)$ with $a<0.$ Then,
\begin{align*}
f_{2,1}(x)&=\pm x \mp x^{2r+1}+c\,x^{4r+1}+a\,x^{2n+1}\,(\pm
1\mp x^{2r}+c\,x^{4r})^{2n+1}
+O(\min(2r+2,2n+2))\\
{ }&=\pm x \mp x^{2r+1} \pm a\,x^{2n+1}+O(\min(2r+2,2n+2)).
\end{align*}

When $f_1$ also preserves orientation,  $f_{2,1}(x)= x - x^{2r+1} +
a\,x^{2n+1}+O(\min(2r+2,2n+2)).$ Then, if $r\ne n,$ from
Theorem~\ref{T-orient-pres} the origin is LAS for the composition
map. If $r=n$ then $f_{2,1}(x)=x+(a-1)\,x^{2n+1}+O(2n+2)$ and since
$a-1<0,$ applying again Theorem~\ref{T-orient-pres}  the result
follows.

When $f_1$ reverses orientation, $f_{2,1}(x)= -x + x^{2r+1} -
a\,x^{2n+1}+O(\min(2r+2,2n+2)).$ In this case
\[
f_{2,1}^2(x)= x - 2x^{2r+1} + 2 a\,x^{2n+1}+O(\min(2r+2,2n+2)).
\]
Applying the same tools that in the previous situation, but to
$f_{2,1}^2,$ the result also follows.

\bigskip

$(b)$ In this case, without loss of generality, we consider $f_1$
written in normal form $f_1(x)=-x+x^{2r+1}+c\,x^{4r+1}$. We also
consider $f_2(x)=-x+\sum_{k\geq 2} a_k x^k$ such that either $V_3<0$
or $V_{2j+1}=0$ for $j=1,2,\ldots,m-1$ for $m\geq 2$, and
$V_{2m+1}<0.$ Now we split the proof in  three subcases:

\bigskip

$(i)$ Assume first that $V_3=-2(a_2^2+a_3)<0.$ Then
$f_2(x)=-x+a_2x^2+a_3x^3 +O(4),$ with $a_2^2+a_3>0$ and
\begin{align*}
 f_{2,1}(x)&=x-x^{2r+1}-c\,x^{4r+1}+a_{2}\,x^{2}\left(1-x^{2r}-c\,x^{4r}\right)^{2}
 -a_{3}\,x^{3}\left(1-x^{2r}-c\,x^{4r}\right)^{3}+O(4)\\
\phantom{asd}&=x+a_{2}\,x^{2}-a_{3}\,x^{3}-x^{2r+1}+O(4).
\end{align*}
By Theorem~\ref{T-orient-pres}, when $a_2\ne0$ the origin is
semi-AS. If $a_2=0$, then $a_3>0$ and the origin is LAS for all
$r\ge1.$

\bigskip

$(ii)$ In this second case we suppose that $V_3=0$ and that there
exists $1\le j\le m-1$ such that $a_{2i}=0$ for $i=1,2,\ldots,j-1$
and $a_{2j}\ne 0.$ Since $V_{2i+1}=0$ for $i=1,2,\ldots,j-1$
applying Lemma \ref{L-lema2} we have that $a_{2i+1}=0$ for
$i=1,2,\ldots, j-1.$ Hence $f_2(x)=-x+a_{2j}\,x^{2j}+O(2j+1),$ and
\begin{align*}
 f_{2,1}(x)&=x-x^{2r+1}-c\,x^{4r+1}+a_{2j}\,x^{2j}\left(1-x^{2r}-c\,x^{4r}\right)^{2j}+O(\min(2r+2,2j+1))\\
\phantom{asd}&=x-x^{2r+1}+a_{2j}\,x^{2j}+O(\min(2r+2,2j+1)).
\end{align*}
Hence, by Theorem~\ref{T-orient-pres}, if $j\le r$ then the origin
is a semi-AS fixed point and if $j>r$ then it is a LAS fixed point.

\bigskip

$(iii)$  Finally  assume that $V_3=0$ and   that $a_{2i}=0$ for
$i=1,2,\ldots,m-1.$ Since $V_{2j+1}=0,$ for $j=1,2,\ldots,m-1$, and
$V_{2m+1}<0,$ from Lemma \ref{L-lema2} we get that $a_{2i+1}=0$ for
$i=1,2,\ldots,m-1.$ Moreover $m\ge2.$ Consequently
$f_2(x)=-x+a_{2m}\,x^{2m}+a_{2m+1}\,x^{2m+1}+O(2m+2)$. Then,
$f_2^{2}(x)=x-2\,a_{2m+1}\,x^{2m+1}+O(2m+2),$ and therefore
$0>V_{2m+1}=-2\,a_{2m+1}.$ Moreover,
\begin{align*}
 f_{2,1}(x)=&x-x^{2r+1}-c\,x^{4r+1}+a_{2m}\,x^{2m}\left(1-x^{2r}-c\,x^{4r}\right)^{2m}-\\&
 a_{2m+1}\,x^{2m+1}\left(1-x^{2r}-c\,x^{4r}\right)^{2m+1}+O\big(\min(2r+2,2m+2)\big)\\
\phantom{asd}=&x-x^{2r+1}+a_{2m}\,x^{2m}-a_{2m+1}\,x^{2m+1}+O\big(\min(2r+2,2m+2)\big).
\end{align*}

Assume first that $a_{2m}\ne0.$ Then, again by
Theorem~\ref{T-orient-pres}, when $2r+1>2m$ then the origin is
semi-AS and when $2r+1<2m$  the origin is LAS.

Finally, suppose that  $a_{2m}=0.$ Applying once more
Theorem~\ref{T-orient-pres} we obtain that in all cases the origin
is LAS because $a_{2m+1}>0.$~\end{proof}

\begin{proof}[Proof of Theorem \ref{T-laslasreprep}]

Statement $(a)$ is a direct consequence of Theorem \ref{T-main}.

 $(b)$ Observe that  the maps  of Examples~\ref{E-f1f2f3} and~\ref{E-F1F2F3}
 prove the statement for $k=3$. Indeed, for instance, remember that the
  maps of Example~\ref{E-f1f2f3} have been chosen
in such a way that the stability constats satisfy $V_3(f_i)=0$ and
$V_5(f_i)<0$ for $i\in\{1,2,3\}$, so that they have a LAS fixed
point at the origin. The stability of the origin for these maps can
also be straightforwardly obtained from the first terms of the
Taylor series at the origin of $f_i^{2},i=1,2.$ A computation gives
$ f_{3,2,1}(x)=-x+90x^4-48x^5+O(6), $ so from
Proposition~\ref{P-stabil-const} we have $V_3(f_{3,2,1})=0$ and
$V_5(f_{3,2,1})=96>0$, and therefore the origin is a repeller fixed
point of $f_{3,2,1}.$ Since $f_{3,2,1}^{2}(x)=x+96x^5+O(7),$ the
result also follows from Theorem \ref{T-orient-pres}.

Next we show how to construct the maps of Example~\ref{E-f1f2f3}. We
start with some maps
$$
f_i(x)=-x+a_{2,i}x^2+a_{3,i}x^3+a_{4,i}x^4+a_{5,i}x^5,\,
i\in\{1,2,3\}.
$$
To get that $V_3(f_i)=0$ and $V_5(f_i)<0$, first we take
$a_{3,i}=-a_{2,i}^2$, and then impose that $V_5(f_i)=2\left(2
a_{2,i}^{4}-3 a_{2,i}a_{4,i}-a_{5,i}\right)=-2A_i^2$, obtaining $
a_{5,i}=2a_{2,i}^4-3a_{2,i}a_{4,i}+A_i^2.$

At this point we notice that
$$
f_{3,2,1}(x)=-x+(a_{2,1}+a_{2,3}-a_{2,2}^2)x^2-(a_{2,1}+a_{2,3}-a_{2,2}^2)^2x^3+O(4),
$$
hence $V_3(f_{3,2,1})=0$. In order to reduce parameters and simplify
the expressions we take $a_{2,1}=a_{2,2}^2-a_{2,3}$, obtaining
$f_{3,2,1}(x)=-x+(3a_{2,2}^2a_{2,3}-3a_{2,2}a_{2,3}^2+a_{4,1}-a_{4,2}+a_{4,3})x^4+O(5)$.
Again, to reduce parameters we take $a_{2,3}=2$ and
$a_{4,1}=a_{4,2}=a_{4,3}=0$. With this choice we get that
$$
V_5(f_{3,2,1})=-2\left(A_1^2+A_2^2+A_3^2-4a_{2,2}^3+24a_{2,2}^2-32a_{2,2}\right).
$$
Taking $A_1=\sqrt{2}$, $A_2=3$ and $A_3=1$, we obtain that
$$
V_5(f_{3,2,1})=-2\left(-4a_{2,2}^3+24a_{2,2}^2-32a_{2,2}+12\right)=8(a_{2,2}-1)(a_{2,2}^2-5a_{2,2}+3),
$$ and therefore $V_5(f_{3,2,1})>0$ if and only if
$a_{2,2}\in\left((5-\sqrt{13})/2,1\right)$ or
$a_{2,2}>(5+\sqrt{13})/2\simeq 4.303$. Taking $a_{2,2}=5$, we obtain
the maps of Example~\ref{E-f1f2f3}.

Now we consider the case $k>3$. We take the maps $f_1,f_2$  and
$f_3$ given in Example~\ref{E-f1f2f3}, and for all
$j\in\{4,\ldots,k\}$ we consider the maps $f_j(x)=x- x^{7}, $  so
that all them have a LAS fixed point at the origin (by Theorem
\ref{T-orient-pres}). Then $f_{k,\ldots,4}(x)=x-(k-3)x^7+O(13)$.

Observe that if we take any map of the form
$g(x)=-x+\sum_{j=2}^{5}\alpha_j x^j +O(6)$ we obtain
$f_{k,\ldots,4}\circ g(x)=-x+\sum_{j=2}^{5}\alpha_j x^j +O(6)$. Thus
$f_{k,\ldots,1}(x)=f_{k,\ldots,4}\circ
f_{3,2,1}(x)=-x+90x^4-48x^5+O(6).$ Therefore $V_3(f_{k,\ldots,1})=0$
and $V_5(f_{k,\ldots,1})=96>0$, and  the origin is  repeller for
$f_{k,\ldots,1}$.

 Similarly, consider the maps $g_1$, $g_2$ and $g_3$
given in Example \ref{E-F1F2F3} and $g_j(x)=x+x^7$ for all
$j\in\{4,\ldots,k\}$. By construction, each map $g_j$ has a repeller
fixed  point at the origin. Now
$g_{k,\ldots,1}(x)=-x+90x^4+48x^5+O(6), $ and a computation shows
that $V_3(g_{k,\ldots,1})=0$ and $V_5(g_{k,\ldots,1})=-96<0.$ Hence
the origin is  LAS for $g_{k,\ldots,1}$.
\end{proof}

\section{Proof of Theorems
\ref{T-atratr-dim2} and \ref{T-Main-A}}\label{S-dim-2}

We start recalling the tools that we will use to know the stability
of the elliptic fixed point.

\subsection{Birkhoff normal form and stability}\label{ss:kam}

We remark that in this article we are not interested only on the
\emph{stability} of the elliptic fixed points, that is one of the
issues  that people usually refers in the context of studying maps
with elliptic points, via Moser twist theorem and KAM theory. Here
we want to know whether these fixed points are LAS or repeller. This
information is given by the so called \emph{Birkhoff constants} that
we recall next.

Given an elliptic fixed point with eigenvalues $\lambda,\bar
\lambda=1/\lambda,$ that are not roots of unity of order $\ell$ for
$0<\ell\leq 2m+1,$ we will say that $p$ is a \emph{non
$(2m+1)$-resonant} elliptic point. Near a non $(2m+1)$-resonant
elliptic fixed point, a $\mathcal{C}^{2m+2}$-map $f$ is locally
conjugated to its \emph{{Birkhoff normal form}} plus some remainder
terms, see \cite{AP}. This normal form is
\begin{equation*}
f_{B}(z,\bar z)=\lambda z\Big(1+\sum\limits_{j=1}^{m} B_j
(z\bar{z})^j\Big)+O(2m+2),
\end{equation*}
where $z=x+y i$ and $B_j$ are complex numbers. The first
non-vanishing number $B_j$ is called the $j$th \emph{Birkhoff
constant}. Then, $V_j=\mathrm{Re}(B_j)$ will be called the
$j$th \emph{Birkhoff stability constant}. Both constants provide very
useful dynamical information in a neighborhood of the elliptic
point. In this sense, a well-known result is the following:

\begin{lem}[\cite{CGM15}]\label{P-dynamicaFN-gen}
For $m\in\N$, consider a $\mathcal{C}^{2m+2}$-map $f$  with an
elliptic fixed point $p\in \mathcal{U}$,  non $(2m+1)$-resonant. Let
$B_m$ be its first non-vanishing  Birkhoff constant. If
$V_m=\mathrm{Re}(B_m)< 0$ (resp. $V_m=\mathrm{Re}(B_m)>0$), then the
point $p$ is  LAS (resp. repeller).
\end{lem}

The computation of the Birkhoff normal forms and the corresponding
Birkhoff constants is a subject on which there is abundant
literature, the reader is referred  for instance to~\cite{AP}. Next
we give the expression of the first Birkhoff constant of a map with
a non $3$-resonant fixed point at the origin. Set
\begin{equation*}%\label{E-F-diagonal}
g(z,\bar z)=\lambda z\,+\sum_{m+j=2}^3
a_{m,j}z^m\bar{z}^j+O(4),\end{equation*}
where $z\in\mathbb{C}$,
$\lambda\in\C,$ $|\lambda|=1$. Then the first Birkhoff constant is
\begin{equation}\label{E-B1-Gen}
B_1=B_1(g)=\frac{ P(g) }{{\lambda}^{2} \left( \lambda-1 \right)
\left( {\lambda}^{2}+\lambda+ 1 \right) },
\end{equation}
where
\begin{align*}
         P(g)=& \left( |a_{11}|^2+a_{21} \right)
{\lambda}^{4}-a_{11}
 \left( 2a_{20}-\overline{a}_{11} \right) {\lambda}^{3}+ \left( 2
|a_{02}|^2 -a_{11}a_{20}+|a_{11}|^2 \right){\lambda}^{2}\\
&- \left( a_{11}a_{20}+a_{21} \right) \lambda+ a_{11}a_{20},
        \end{align*}
see for instance~\cite[Sec. 4]{CGM13}.

Lemma \ref{P-dynamicaFN-gen} allows us to utilize the Birkhoff
stability constants in an analogous way as we used the stability
constants for one-dimensional map. Therefore we follow a similar
idea than the used to construct the maps in the proof of Theorem
\ref{T-laslasreprep}, in order to prove Theorem \ref{T-atratr-dim2}:
we will construct two maps $f_1$ and $f_2$ such that $V_1(f_i)<0$
(resp. positive) and $V_1(f_{2,1})>0$ (resp. negative).

\subsection{Proof of Theorem \ref{T-atratr-dim2}}\label{proofTB}

Consider the maps $f_1$ and $f_2$  of Example \ref{Ex-dim2-1}. To
compute their Birkhoff constants, first we write them in complex
notation obtaining that
\begin{equation}\label{E-G-LAS-LAS-REP-2-DIM}
g_1(z,\bar{z})=iz+(1-3i)z^2+z\bar{z}\,\mbox{ and }\,
g_2(z,\bar{z})=\left(\frac{1}{2}+\frac{\sqrt{3}}{2}i\right)z-z^2\bar{z},
\end{equation} are their respective equivalent complex expressions.
Conversely, the real expressions of each $g_j(x,y)$ are obtained
taking
\begin{equation}\label{E-complex-a-reals}
f_j(x,y)=\left(\mathrm{Re}\left(g_j(x+y i,x-y i)\right),
\mathrm{Im}\left(g_j(x+y i,x-y i)\right)\right), \quad j=1,2.
\end{equation}
Next, we apply the formula (\ref{E-B1-Gen}) to the maps $
g_j(z,\bar{z})$, $j=1,2.$  We get that
 their first Birkhoff constants are
$$
B_1(g_1)=-\frac{1}{2}-\frac{11}{2}\, i \quad\mbox{ and
}\quad B_1(g_2)=-\frac{1}{2}+\frac{\sqrt{3}}{2}\, i.
$$
% Notice that $B_1(f_i)=B_1(g_i)$ for $i=1,2$, 
So their Birkhoff
stability constants are $V_1(g_1)=V_1(g_2)=-1/2<0.$ By Lemma
\ref{P-dynamicaFN-gen}, the origin is a LAS fixed point for both
maps $g_1$ and $g_2$.  Since for each $j=1,2$, 
$f_{j}(x,y)$ and $g_{j}(z,\bar z)$ are different expressions of
the same map, the origin is  LAS  for both maps $f_1$ and $f_2$.

The composition map
$$
g_{2,1}(z,\bar{z})=g_2\circ g_1(z,\bar{z})=i\,\left(\frac{1}{2}+\frac{\sqrt{3}}{2}
i\right)z+
\frac{1}{2}(1-3i)(1+\sqrt{3}i)z^2+\frac{1}{2}(1+\sqrt{3}i)z\bar{z}-iz^2\bar{z}+O(4),
$$
has an associated Birkhoff constant
\begin{equation*}\label{E-valor-B1F21}
B_1(g_{2,1})=\frac{3\sqrt{3}-5}{2}+i\,\frac{3\sqrt{3}-13}{2}\simeq
0.098-3.902\,i,
\end{equation*}
and therefore  $V_1(g_{2,1})>0$. So again by Lemma \ref{P-dynamicaFN-gen}
the origin is a repeller fixed point for $g_{2,1}(z,\bar{z})$, hence also 
for the composition map
$f_{2,1},$ as we wanted to prove.

%Since
%$f_{2,1}(x,y)$ and $g_{2,1}(z,\bar z)$ are different expressions of
%the same map, the origin is  repeller  

Next we will explain how we have found the Example \ref{Ex-dim2-1},
used in the above proof. We start with
$$
g_1(z,\bar z)=\alpha z+\sum\limits_{m+j=2}^{3}a_{m,j}z^m
\bar{z}^j\quad \mbox{ and }\quad g_2(z,\bar z)=\beta
z+c_{2,1}z^2\bar z.
$$
The last map has been chosen so that it only contains the cubic
resonant terms that appear in the formula (\ref{E-B1-Gen}). We take
$\alpha=i$ and $\beta=(1+\sqrt{3} \,i)/2$, so that the origin is non
a $3$-resonant elliptic fixed point for both maps. We compute the
Birkhoff constant, using the formula (\ref{E-B1-Gen}), obtaining
\begin{align*}
B_1(g_1)&=\frac{1-i}{2} \left( -2|a_{0,2}|^2+2a_{1,1}a_{2,0}+a_{2,1}
+ \left( a_{1,1} a_{2,0}-|a_{1,1}|^2-a_{2,1} \right) i\right),\\
B_1(g_2)&=\frac{1-i\sqrt{3}}{2}c_{2,1} \end{align*} and
\begin{align*}
    B_1(g_{2,1})=& -\frac{1}{2} |a_{1,1}|^2  - |a_{0,2}|^2+\frac{3}{2} a_{1,1} a_{2,0}+\frac{1}{2}
c_{2,1}\\ &+ \Big(\frac{\sqrt {3 }}{2} a_{1,1} a_{2,0} +\frac{\sqrt
{3}}{2} |a_{1,1}|^2 -\frac{\sqrt {3}}{2} c_{2,1} -
|a_{0,2}|^2-a_{1,1} a_{2,0} -|a_{1,1}|^2 -a_{2,1} \Big) i.
  \end{align*}

In order to reduce the parameters we set $ a_{2,0}=t+s i$,
$a_{11}=1$, $a_{0,2}=0$, $a_{2,1}=0$, $c_{2,1}=u$  where
$s,t,u\in\R$. We get:
\begin{equation}\label{E-B1F-i-G}
B_1(g_1)=\frac{1}{2}\left( 3t+s-1+(-t+3s-1) i\right),\quad
B_1(g_2)=\frac{1}{2}\left(1-\sqrt{3}\,i\right)u,
\end{equation}
 and
\begin{equation}\label{E-B1GF}
B_1(g_{2,1})= \left(1-\frac{\sqrt
{3}}{2}\right)\,s+\frac{3}{2}\,t+\frac{1}{2}\,u-\frac{1}{2}+\left(
\frac{3}{2}\,s+\left(\frac{\sqrt {3}}{2}-1\right)\,t-\frac{\sqrt
{3}}{2}\,u-1+\frac{\sqrt {3}}{2}\,
 \right) i
\end{equation}

To simplify more the above expressions, we consider $s=-3t$ and
$u=-1$, obtaining that $V_1(g_1)=V_1(g_2)=-1/2$, and
$$
V_1(g_{2,1})=-1+\frac{3}{2}(\sqrt{3}-1)\,t.
$$
This last constant is positive for all $t>2/(3(\sqrt{3}-1))\simeq
0.911$, so taking $t=1$ we get the maps
(\ref{E-G-LAS-LAS-REP-2-DIM}). Applying the formula
(\ref{E-complex-a-reals}) we obtain  the expression of the maps of
Example~\ref{Ex-dim2-1}.

We have already commented before Example \ref{Ex-dim2-2} that from
the above example, simply taking the inverse maps we could construct
two planar maps having a common repeller fixed point such that the
corresponding composition map has a LAS fixed point. Nevertheless,
next we construct a simple explicit example, namely, Example
\ref{Ex-dim2-2}. In this case, to reduce parameters in the
expressions (\ref{E-B1F-i-G}) and~(\ref{E-B1GF}), we take $s=2-3t$
and $u=1$, obtaining that $V_1(g_1)=V_1(g_2)=1/2$, and
$$
V_1(g_{2,1})=2-\sqrt{3}+\frac{3}{2}(\sqrt{3}-1)\,t.
$$
This constant is negative for all $t<\frac{1}{3}\, \left(\sqrt {3}
-2 \right)  \left( 1+\sqrt {3} \right) \simeq -0.244$. Setting
$t=-2/3$, and applying the formula (\ref{E-complex-a-reals}) we get
the maps of this last example.

\subsection{Proof of Theorem
\ref{T-Main-A}} \label{S-Main-results-n-dim}

$(a)$ Consider the $k\ge3$ functions $\{f_1,f_2,\ldots,f_k\}$  used
for proving item $(b)$ of Theorem~\ref{T-laslasreprep},  that is
$f_1, f_2$ and $f_3$ given in Example \ref{E-f1f2f3}, and $f_j(x)=x-x^7$
for $j=4,\ldots,k.$  For any $n\in\N,$  define the maps
\[
F_j\big(x_1,x_2,\ldots,x_n\big)=\big(f_j(x_1),f_j(x_2),\ldots,f_j(x_n)\big),\quad
j=1,2,\ldots,k.
\]
which are from $\R^n$ into itself. Because the components of the
above maps are uncoupled, from Theorem~\ref{T-laslasreprep} we
obtain that the origin is a LAS fixed  point for each $F_j$, but a
repeller fixed point for $F_{k,\ldots,1},$

Analogously, we take the maps  $g_j, j=1,\ldots,k,$ given at the end
of the proof of Theorem~\ref{T-laslasreprep}, and define
$G_j\big(x_1,x_2,\ldots,x_n\big)=\big(g_j(x_1),g_j(x_2),\ldots,g_j(x_n)\big)$,
for $j=1,2,\ldots,k$. For these maps, the origin is a repeller fixed
point for each  $G_j$ but it is LAS fixed point for
$G_{k,\ldots,1}$.

(b) For any $m\in\N$ we define the two maps from
$R^{2m}$ into itself,
\[
F_j\big(x_1,x_2,\ldots,x_{2m}\big)=\big(f_j(x_1,x_2),f_j(x_3,x_4),
\ldots,f_j(x_{2m-1},x_{2m})\big),\quad
j=1,2,
\]
where the $f_j$ are the ones appearing either in
Example~\ref{Ex-dim2-1} or in Example~\ref{Ex-dim2-2}. Then the
result follows taking the periodic set $\{F_1,F_2\}$  and noticing
that the dynamics of each consecutive pair of components of any map
$F_j$ is uncoupled.

\subsection*{Acknowledgements} The authors are supported by
Ministry of Economy, Industry and Competitiveness of the Spanish
Government through grants MINECO/FEDER MTM2016-77278-P  (first and
second authors) and DPI2016-77407-P (third author). The first  and
second authors are also supported by the grant 2014-SGR-568  from
AGAUR, Generalitat de Catalunya. The third author is supported by
the grant 2014-SGR-859 from AGAUR, Generalitat de Catalunya.

\end{document}